\documentclass{amsart}
\usepackage{amsmath}
\usepackage{amssymb}
\usepackage{amsthm}

\DeclareMathOperator{\tr}{tr}

\DeclareMathOperator{\dvol}{dvol}

\DeclareMathOperator{\Ric}{Ric}

\DeclareMathOperator{\Rm}{Rm}

\newcommand{\cv}{\widetilde{v}}

\newcommand{\cB}{\widetilde{B}}

\newcommand{\cE}{\widetilde{E}}

\newcommand{\cR}{\widetilde{R}}

\newcommand{\cT}{\widetilde{T}}

\newcommand{\csigma}{\widetilde{\sigma}}

\newcommand{\hsigma}{\widehat{\sigma}}

\newcommand{\lp}{\langle}
\newcommand{\rp}{\rangle}
\newcommand{\lv}{\lvert}
\newcommand{\rv}{\rvert}

\newcommand{\eps}{\varepsilon}

\newcommand{\mC}{\mathcal{C}}

\newcommand{\mF}{\mathcal{F}}

\newcommand{\mV}{\mathcal{V}}
\newcommand{\mW}{\mathcal{W}}

\newcommand{\bN}{\mathbb{N}}

\newcommand{\bR}{\mathbb{R}}

\newcommand{\cRic}{\widetilde{\Ric}}

\def\sideremark#1{\ifvmode\leavevmode\fi\vadjust{\vbox to0pt{\vss
 \hbox to 0pt{\hskip\hsize\hskip1em
 \vbox{\hsize3cm\tiny\raggedright\pretolerance10000
 \noindent #1\hfill}\hss}\vbox to8pt{\vfil}\vss}}}

\newcommand{\suchthat}{\mathrel{}\middle|\mathrel{}}

\newcommand{\comment}[1]{}

\newtheorem{thm}{Theorem}[section]

\newtheorem{lem}[thm]{Lemma}
\newtheorem{cor}[thm]{Corollary}

\theoremstyle{definition}

\theoremstyle{remark}

\numberwithin{equation}{section}

\begin{document}

\title{A weighted renormalized curvature for manifolds with density}
\author{Jeffrey S. Case}
\address{109 McAllister Building \\ Penn State University \\ University Park, PA 16802}
\email{jscase@psu.edu}
\keywords{smooth metric measure space, manifold with density, renormalized volume, gradient Ricci soliton, $\mW$-functional}
\subjclass[2010]{Primary 53C21; Secondary 53C25, 58E11}
\begin{abstract}
We introduce a scalar invariant on manifolds with density which is analogous to the renormalized volume coefficient $v_3$ in conformal geometry.  We show that this invariant is variational and that shrinking gradient Ricci solitons are stable with respect to the associated $\mW$-functional.
\end{abstract}
\maketitle

\section{Introduction}
\label{sec:intro}

The $\sigma_k$-curvatures of a Riemannian manifold $(M^n,g)$ are defined as the $k$-th elementary symmetric functions of the eigenvalues of the Schouten tensor.  These invariants have been heavily studied in conformal geometry, in part because of their variational properties, especially as related to Einstein metrics~\cite{BransonGover2008,Viaclovsky2000}.  In this respect, Chang, Fang and Graham observed~\cite{ChangFang2008,ChangFangGraham2012,Graham2009} that the renormalized volume coefficients $v_k$~\cite{Graham2000}, which agree with the $\sigma_k$-curvature when the latter is variational or when the metric is Einstein~\cite{Graham2009,GrahamJuhl2007}, have even better properties: the renormalized volume coefficients are always variational~\cite{ChangFang2008,Graham2000}; have highest order term equal to the $\sigma_k$-curvature~\cite{Graham2009}; and Einstein metrics with nonzero scalar curvature are stable with respect to $\mF_k$~\cite{ChangFangGraham2012,GuoLi2011}.

Motivated by the question of finding invariants with similar properties with respect to gradient Ricci solitons, the author introduced the notion of the weighted $\sigma_k$-curvatures on a manifold with density~\cite{Case2014sd}.  A manifold with density is a triple $(M^n,g,e^{-\phi}\dvol)$ consisting of a Riemannian manifold and a smooth measure; such a space is a gradient Ricci soliton if $\Ric+\nabla^2\phi=\lambda g$ for some $\lambda\in\bR$.  The weighted $\sigma_k$-curvatures $\csigma_{k,\phi}$ are defined for all $k\in\bN$ and $\lambda\in\bR$.  These invariants are especially useful for studying gradient Ricci solitons with $\Ric+\nabla^2\phi=\lambda g$.  The first three such invariants are
\begin{align*}
 \csigma_{1,\phi} & := \frac{1}{2}\left(R + 2\Delta\phi - \lv\nabla\phi\rv^2 + 2\lambda(\phi-n)\right), \\
 \csigma_{2,\phi} & := \frac{1}{2}\left(\left(\csigma_{1,\phi}\right)^2 - \lv\cRic_\phi\rv^2\right), \\
 \csigma_{3,\phi} & := \frac{1}{6}\left(\left(\csigma_{1,\phi}\right)^3 - 3\csigma_{1,\phi}\lv\cRic_\phi\rv^2 + 2\tr\bigl(\cRic_\phi\bigr)^3\right),
\end{align*}
where $\cRic_\phi=\Ric+\nabla^2\phi-\lambda g$ and $\bigl(\cRic_\phi\bigr)^3$ is the three-fold composition of $\cRic_\phi$ with itself when regarded as an endomorphism of $TM$ via the metric.  Analogous to the Riemannian setting, the weighted $\sigma_k$-curvatures are variational if and only if $k\in\{1,2\}$ or $g$ is flat, and non-steady gradient Ricci solitons are stable with respect to the corresponding functionals in these cases; see~\cite{Case2014sd} for details.  One would like to remove the restriction that $g$ be flat by introducing a family of ``weighted renormalized volume coefficients'' on manifolds with density which are variational and recover the weighted $\sigma_k$-curvatures for flat metrics.

The purpose of this note is to provide evidence that such a family exists.  Specifically, we show that the \emph{(third) renormalized volume coefficient}
\[ \cv_{3,\phi} := \csigma_{3,\phi} + \frac{1}{3}\lp\cB_\phi,\cRic_\phi\rp \]
is the correct weighted analogue of the corresponding Riemannian invariant $v_3$.  Here $\cB_\phi$ is the weighted Bach tensor
\begin{equation}
 \label{eqn:weighted_bach}
 \cB_\phi := \delta_\phi d\cRic_\phi + \Rm\cdot\cRic_\phi ;
\end{equation}
see Section~\ref{sec:bg} for an explanation of our notation.  Note that if $g$ is flat or $\cRic_\phi=0$, then $\cB_\phi=0$.  In particular, $\cv_{3,\phi}=\csigma_{3,\phi}$ on flat manifolds and gradient Ricci solitons.  For this reason, and in analogy with the Riemannian situation (cf.\ \cite{Graham2009}), we define $\cv_{1,\phi}:=\csigma_{1,\phi}$ and $\cv_{2,\phi}:=\csigma_{2,\phi}$ as the first and second renormalized volume coefficients.  It is easy to check that
\[ \cB_\phi = 2d\Ric(\cdot,\nabla\phi,\cdot) + \Rm\cdot\left(\Ric+d\phi\otimes d\phi\right) + \delta d\Ric - \lambda\Ric , \]
and hence, for $(g,\tau)$ fixed, $\cv_{k,\phi}$ depends only on the two-jet of $\phi$ for all $k\in\{1,2,3\}$.

The variational properties of $\cv_{3,\phi}$ can be understood in terms of the $\mW_3$-functional
\begin{equation}
 \label{eqn:W3}
 \mW_3(g,\phi,\tau) := \int_M \left[ \tau^3\cv_{3,\phi} + \frac{\tau}{8}\cv_{1,\phi} \right] (4\pi\tau)^{-\frac{n}{2}}e^{-\phi}\dvol_g
\end{equation}
defined on a compact manifold $M^n$, where $\cv_{3,\phi}$ and $\cv_{1,\phi}$ are defined in terms of the manifold with density $(M^n,g,e^{-\phi}\dvol)$ and the parameter $\lambda=1/(2\tau)>0$.  It is clear that the $\mW_3$-functional is scale-invariant: $\mW_3(cg,\phi,c\tau)=\mW_3(g,\phi,\tau)$ for all $c>0$.  We are interested in the first and second variation of this functional in the scale-invariant classes
\begin{align*}
 \mC_1(g,\tau) & := \left\{ \phi\in C^\infty(M) \suchthat \int_M (4\pi\tau)^{-\frac{n}{2}}e^{-\phi}\dvol_g = 1 \right\} , \\
 \mC_1(g) & := \left\{ (\phi,\tau)\in C^\infty(M)\times(0,\infty) \suchthat \phi\in\mC_1(g,\tau) \right\} .
\end{align*}

Our first result is a characterization of the critical points of $\mW_3\colon\mC_1(g,\tau)\to\bR$; in particular, this shows that $\cv_{3,\phi}$ is variational.

\begin{thm}
 \label{thm:variational_theorem}
 Let $(M^n,g)$ be a Riemannian manifold and fix $\tau>0$.  Then $\phi$ is a critical point of $\mW_3\colon\mC_1(g,\tau)\to\bR$ if and only if there is a constant $c\in\bR$ such that
 \begin{equation}
  \label{eqn:variational_theorem}
  \cv_{3,\phi} - \frac{1}{2\tau}\cv_{2,\phi} + \frac{1}{8\tau^2}\cv_{1,\phi} = c .
 \end{equation}
\end{thm}

Using~\cite[Proposition~5.2 and Proposition~5.7]{Case2014sd}, it is easy to check that the critical points of $\mW_3+\frac{1}{2}\mW_2+\frac{1}{8}\mW_1$ in $\mC_1(g,\tau)$ are such that $\cv_{3,\phi}$ is constant, hence $\cv_{3,\phi}$ is variational.

Our second result shows that shrinking gradient Ricci solitons are stable with respect to $\mW_3\colon\mC_1(g)\to\bR$.  To that end, we say that $(M^n,g,e^{-\phi}\dvol)$ is a \emph{volume-normalized shrinking gradient Ricci soliton} if there is a constant $\tau>0$ such that $\Ric+\nabla^2\phi=\frac{1}{2\tau}g$ and $(\phi,\tau)\in\mC_1(g)$.  Note that if $\Ric+\nabla^2\phi=\frac{1}{2\tau}g>0$, then $e^{-\phi}\dvol$ is a finite measure~\cite{Wylie2008} and hence, by adding a constant to $\phi$ if necessary, we can ensure that $\phi\in\mC_1(g,\tau)$.  To state our result, recall that if $(\phi,\tau)$ is a critical point of $\mW_3\colon\mC_1(g)\to\bR$, then we may define the quadratic form $\mW_3^{\prime\prime}$ on
\[ T_{(\phi,\tau)}\mC_1(g) = \left\{ (\psi,t)\in C^\infty(M)\times\bR \suchthat \int_M \left(\psi+\frac{nt}{2\tau}\right)\,(4\pi\tau)^{-\frac{n}{2}}e^{-\phi}\dvol = 0 \right\} . \]
by $\mW_3^{\prime\prime}(v):=\left(\mW_3(\gamma(s))\right)^{\prime\prime}(0)$ for $\gamma\colon(-\eps,\eps)\to\mC_1(g)$ a smooth path with $\gamma^\prime(0)=v$.

\begin{thm}
 \label{thm:stability_theorem}
 Let $(M^n,g,e^{-\phi}\dvol)$ be a volume-normalized shrinking gradient Ricci soliton.  Then $(\phi,\tau)$ is a critical point of $\mW_3\colon\mC_1(g)\to\bR$ and $\mW_3^{\prime\prime}$ is a positive semi-definite quadratic form on $T_{(\phi,\tau)}\mC_1(g)$.  Moreover, either
 \begin{enumerate}
  \item the null space of $\mW_3^{\prime\prime}$ is $(n+1)$-dimensional, in which case $(M^n,g,e^{-\phi}\dvol)$ is isometric to an $n$-dimensional Gaussian;
  \item the null space of $\mW_3^{\prime\prime}$ is $k$-dimensional, $1\leq k\leq n-1$, in which case $(M^n,g,e^{-\phi}\dvol)$ factors through a $k$-dimensional Gaussian; or
  \item $\mW_3^{\prime\prime}$ is positive definite.
 \end{enumerate}
\end{thm}

Here an $n$-dimensional Gaussian is one of the manifolds with density
\[ \left( \bR^n, dx^2, \exp\left(-\frac{\lv x\rv^2}{4\tau}\right)\dvol \right) \]
for some $\tau>0$.  These are all volume-normalized shrinking gradient Ricci solitons.  We say that $(M^n,g,e^{-\phi}\dvol)$ \emph{factors through a $k$-dimensional Gaussian} if it is isometric to a product
\[ \left( N^{n-k}\times\bR^k, h \oplus dx^2, \exp\left(-f-\frac{\lv x\rv^2}{4\tau}\right)\dvol \right) \]
of a manifold with density $(N^{n-k},h,e^{-f}\dvol)$ and a $k$-dimensional Gaussian but it is not isometric to a corresponding product of an $(n-k-1)$-dimensional manifold with density and a $(k+1)$-dimensional Gaussian.  The null directions in Theorem~\ref{thm:stability_theorem} correspond to translations in the Euclidean factors; when $(M^n,g,e^{-\phi}\dvol)$ is isometric to an $n$-dimensional Gaussian, the additional null direction corresponds to the homothety invariance of the Euclidean metric and the scale invariance of the $\mW_3$-functional.  Specifically, if $x_j$ denote the standard coordinates in the Euclidean factor, then $(x_j,0)$ is always a null direction for $\mW_3^{\prime\prime}$, and on the $n$-dimensional shrinking Gaussian, $(\lv x\rv^2,-4\tau^2)$ is also a null direction for $\mW_3^{\prime\prime}$.

Theorem~\ref{thm:stability_theorem} generalizes~\cite[Theorem~9.2]{Case2014sd} and~\cite[Theorem~1.3]{Case2014sd} for the $\mW_1$- and $\mW_2$-functionals, respectively.  Theorem~\ref{thm:stability_theorem} is also the primary reason we define the $\mW_3$-functional by integrating $\tau^3\cv_{3,\phi}+\frac{\tau}{8}\cv_{1,\phi}$: Consider the functional
\begin{equation}
 \label{eqn:V3}
 \mV_3(g,\phi,\tau) := \int_M \tau^3\cv_{3,\phi}\,(4\pi\tau)^{-\frac{n}{2}}e^{-\phi}\dvol_g .
\end{equation}
If $(M^n,g,e^{-\phi}\dvol)$ is a volume-normalized shrinking gradient Ricci soliton which is not isometric to a shrinking Gaussian, then $\mV_3^{\prime\prime}$ determines a positive semi-definite quadratic form on $T\mC_1(g)$ with finite-dimensional kernel.  However, on shrinking Gaussians, $\mV_3^{\prime\prime}=0$ (cf.\ Theorem~\ref{thm:V3_second_var}).  Adding any positive multiple of the $\mW_1$-functional to $\mV_3$ yields a functional whose second variation on shrinking Gaussians is a positive semi-definite quadratic form on $T\mC_1(g)$ with finite dimensional kernel.  We have chosen the particular multiple in our definition~\eqref{eqn:W3} to ensure that the Euler equation~\eqref{eqn:variational_theorem} is equivalent to $\hsigma_{3,\phi}$ being constant on a Euclidean background (cf.\ \cite[Theorem~7.5]{Case2014sd}).

One could conceivably define the fourth weighted renormalized volume coefficient via an explicit local formula (cf.\ \cite[(2.23)]{Graham2009}) and prove analogous results.  However, this is impractical as a means to defining the full family of weighted renormalized volume coefficients, not in the least because explicit formulae for the general renormalized volume coefficients are not known.  A better approach is to develop a suitable notion of the ``weighted ambient metric'' to which one can adapt the arguments used to study the variational properties of the renormalized volume coefficients~\cite{ChangFang2008,ChangFangGraham2012,Graham2009}.
\section{Some preliminary results}
\label{sec:bg}

Before we can investigate variations of the $\mW_3$-functional, we need some facts involving the weighted Bach tensor~\eqref{eqn:weighted_bach} and some estimates involving shrinking gradient Ricci solitons.

Let $(M^n,g,e^{-\phi}\dvol)$ be a manifold with density and fix $\lambda\in\bR$.  Denote $\cRic_\phi$ in abstract index notation by $\cR_{ij}$.  Then the \emph{weighted Cotton tensor} $d\cRic_\phi$ and the weighted Bach tensor $\cB_\phi$ are,
\begin{align*}
 (d\cR)_{ijk} & := \nabla_i\cR_{jk} - \nabla_j\cR_{ik}, \\
 (\cB_\phi)_{ij} & := \nabla^k(d\cR)_{kij} - (d\cR)_{kij}\nabla^k\phi + R_{ikjl}\cR^{kl} ,
\end{align*}
respectively.  It is important for us to compute the weighted divergence
\[ \bigl(\delta_\phi\cB_\phi\bigr)_j := \nabla^k(\cB_\phi)_{kj} - (\cB_\phi)_{kj}\nabla^k\phi \]
of the weighted Bach tensor.

\begin{lem}
 \label{lem:div_bach}
 Let $(M^n,g,e^{-\phi}\dvol)$ be a manifold with density and let $\lambda\in\bR$.  Then
 \[ \delta_\phi\cB_\phi = -d\Ric_\phi\cdot\Ric_\phi, \]
 where $(d\Ric_\phi\cdot\Ric_\phi)_j=(d\cR)_{ijk}\cR^{ik}$.
\end{lem}

\begin{proof}
 \cite[Corollary~5.8]{Case2014sd} implies that
 \[ \delta_\phi\cB_\phi = -\delta_\phi\left(\cE_{2,\phi}-\lambda\cE_{1,\phi}\right) - d\left(\csigma_{2,\phi}-\lambda\csigma_{1,\phi}\right) . \]
 On the other hand, \cite[Proposition~6.2]{Case2014sd} implies that
 \[ \delta_\phi\left(\cE_{2,\phi}-\lambda\cE_{1,\phi}\right) + d\left(\csigma_{2,\phi}-\lambda\csigma_{1,\phi}\right) = d\cRic_\phi\cdot\cRic_\phi . \qedhere \]
\end{proof}

The proofs of Theorem~\ref{thm:variational_theorem} and especially Theorem~\ref{thm:stability_theorem} require three rigidity results for shrinking gradient Ricci solitons.  First, we require the following Lichnerowicz--Obata-type theorem proven by Cheng and Zhou~\cite[Theorem 2]{ChengZhou2013} which holds for manifolds with uniformly positive Bakry-\'Emery Ricci tensor.

\begin{lem}
 \label{lem:grs_obata}
 Let $(M^n,g,e^{-\phi}\dvol)$ be such that $\Ric+\nabla^2\phi\geq\frac{1}{2\tau}g>0$.  Then
 \[ \lambda_1(-\Delta_\phi) := \inf\left\{ \frac{\int \lv\nabla u\rv^2\,e^{-\phi}\dvol}{\int u^2\,e^{-\phi}\dvol} \suchthat u\not=0,  \int_M u\,e^{-\phi}\dvol=0 \right\} \geq \frac{1}{2\tau} . \]
 Equality holds if and only if $(M^n,g,e^{-\phi}\dvol)$ factors through a $k$-dimensional shrinking Gaussian for $k>0$ the dimension of the $L^2$-eigenspace of $-\Delta_\phi$ corresponding to the eigenvalue $\frac{1}{2\tau}$.  This eigenspace is generated by the standard coordinate functions on $\bR^k$.
\end{lem}

Second, we need to know that the weighted $\sigma_k$-curvatures of a volume-normalized shrinking gradient Ricci soliton are constant and have a sign~\cite[Lemma 4.9]{Case2014sd}.

\begin{lem}
 \label{lem:grs_sigmak}
 Let $(M^n,g,e^{-\phi}\dvol)$ be a volume-normalized shrinking gradient Ricci soliton.  Then for all $k\in\bN$,
 \[ \csigma_{k,\phi} = \frac{1}{k!}\left(\csigma_{1,\phi}\right)^k \]
 is constant.  Moreover, $\csigma_{1,\phi}\leq0$ with equality if and only if $(M^n,g,e^{-\phi}\dvol)$ is isometric to a shrinking Gaussian.
\end{lem}

Third, we need to know that the mean-free part $\phi_0$ of the potential $\phi$ of a shrinking gradient Ricci soliton is an eigenfunction for the weighted Laplacian $\Delta_\phi$ and the $L^2$-norm of $\phi_0$ is controlled~\cite[Lemma 4.10]{Case2014sd}.  Then

\begin{lem}
 \label{lem:estimates}
 Let $(M^n,g,e^{-\phi}\dvol)$ be a volume-normalized shrinking gradient Ricci soliton.  Set $\phi_0:=\phi-\frac{n}{2}-2\tau\csigma_{1,\phi}$.
 \begin{align}
  \label{eqn:sigma1_estimate} -\Delta_\phi\phi_0 & = \frac{1}{\tau}\phi_0, \\
  \label{eqn:sigmak_formula} \int_M \phi_0^2 (4\pi\tau)^{-\frac{n}{2}}e^{-\phi}\dvol & \leq \frac{n}{2} .
 \end{align}
 Moreover, equality holds in~\eqref{eqn:sigmak_formula} if and only if $(M^n,g,e^{-\phi}\dvol)$ is isometric to a shrinking Gaussian.
\end{lem}
\section{The first variation of $\mW_3$}
\label{sec:variational}

In order to compute both the first and second variations of $\mW_3$, we need to compute the first variations of $\cv_{k,\phi}$ for $k\in\{1,2,3\}$.  For the remainder of this note, given a Riemannian manifold $(M^n,g)$ and a smooth curve $\gamma(s)=\left(\phi(s),\tau(s)\right)\in C^\infty(M)\times(0,\infty)$, weighted invariants (e.g.\ $\cv_{3,\phi}$) represent the $s$-dependent invariant defined in terms of the manifold with density $(M^n,g,e^{-\phi(s)}\dvol_g)$ with parameter $\lambda=1/\left(2\tau(s)\right)$ and primes (e.g.\ $(\tau^3\cv_{3,\phi})^\prime$) denote the $s$-derivative along the curve.  Similarly, a derivative at zero (e.g.\ $(\tau^3\cv_{3,\phi})^\prime(0)$) denotes the $s$-derivative evaluated at $s=0$.  We denote $(\psi,t)=\frac{\partial\gamma}{\partial s}$ and $\psi_0=\psi+\frac{nt}{2\tau}$.  We always integrate with respect to $d\nu:=(4\pi\tau)^{-\frac{n}{2}}e^{-\phi}\dvol$; in particular, $\frac{\partial}{\partial s}(d\nu) = -\psi_0\,d\nu$.

Recall the first variation of $\cv_{1,\phi}$; see~\cite[(9.2)]{Case2014sd}.

\begin{lem}
 \label{lem:v1_local_all_var}
 Let $(M^n,g)$ be a Riemannian manifold and let $\gamma(s)=\left(\phi(s),\tau(s)\right)$ be a smooth curve in $C^\infty(M)\times(0,\infty)$.  Then
 \begin{equation}
  \label{eqn:v1_local_var}
  \left(\tau\cv_{1,\phi}\right)^\prime = \tau\Delta_\phi\psi_0 + \frac{1}{2}\psi_0 + t\left(\sigma_{1,\phi} - \frac{n}{4\tau}\right) .
 \end{equation}
\end{lem}

Next, we compute the first variation of $\cv_{2,\phi}$ (cf.\ \cite[(9.4)]{Case2014sd}).

\begin{lem}
 \label{lem:v2_local_all_var}
 Let $(M^n,g)$ be a Riemannian manifold and let $\gamma(s)=\left(\phi(s),\tau(s)\right)$ be a smooth curve in $C^\infty(M)\times(0,\infty)$.  Then
 \begin{multline}
  \label{eqn:v2_local_var}
  \left(\tau^2\cv_{2,\phi}\right)^\prime = \tau^2\delta_\phi\left(\cT_{1,\phi}(\nabla\psi_0)\right) + \frac{\tau}{2}\cv_{1,\phi}\psi_0 \\ + t\tau\cv_{1,\phi}\left(\sigma_{1,\phi} - \frac{n}{4\tau}\right) + t\tau\lp \cE_{1,\phi}, \Ric_\phi\rp ,
 \end{multline}
 where $\cT_{1,\phi}=\cv_{1,\phi}g-\cRic_\phi$ and $\cE_{1,\phi}=\cT_{1,\phi}-\cv_{1,\phi}g$.
\end{lem}

\begin{proof}
 Suppose first that $t=0$.  From~\cite[Corollary~8.4]{Case2014sd} we see that
 \begin{equation}
  \label{eqn:v21}
  \left(\tau^2\cv_{2,\phi}\right)^\prime = \tau^2\delta_\phi\left(\cT_{1,\phi}(\nabla\psi_0)\right) + \frac{\tau}{2}\cv_{1,\phi}\psi_0 .
 \end{equation}
 Suppose next that $\psi_0=0$.  The expression
 \[ \tau^2\cv_{2,\phi} = \frac{1}{2}\left(\left(\tau\cv_{1,\phi}\right)^2 - \left| \tau\cRic_\phi\right|^2 \right) \]
 readily implies that
 \begin{equation}
  \label{eqn:v22}
  \left(\tau^2\cv_{2,\phi}\right)^\prime = t\tau\cv_{1,\phi}\left(\sigma_{1,\phi} - \frac{n}{4\tau}\right) - t\tau\lp\cRic_\phi,\Ric_\phi\rp .
 \end{equation}
 Combining~\eqref{eqn:v21} and~\eqref{eqn:v22} yields~\eqref{eqn:v2_local_var}.
\end{proof}

Lastly, we compute the first variation of $\cv_{3,\phi}$.

\begin{lem}
 \label{lem:v3_local_all_var}
 Let $(M^n,g)$ be a Riemannian manifold and let $\gamma(s)=\left(\phi(s),\tau(s)\right)$ be a smooth curve in $C^\infty(M)\times(0,\infty)$.  Then
 \begin{equation}
  \label{eqn:v3_local_var}
  \begin{split}
   \left(\tau^3\cv_{3,\phi}\right)^\prime & = \tau^3\delta_\phi\left(\left(\cT_{2,\phi}+\frac{1}{3}\cB_\phi\right)(\nabla\psi_0)\right) + \frac{\tau^2}{2}\cv_{2,\phi}\psi_0 \\
    & \quad + t\tau^2\cv_{2,\phi}\left(\sigma_{1,\phi}-\frac{n}{4\tau}\right) + t\tau^2\left\lp \cE_{2,\phi} + \frac{1}{3}\cB_\phi,\Ric_\phi\right\rp \\
    & \quad+ \frac{t\tau^2}{3}\left\lp 2\cB_\phi + \frac{1}{2\tau}\Ric, \cRic_\phi\right\rp ,
  \end{split}
 \end{equation}
 where $\cT_{2,\phi}=\cv_{2,\phi}g-\cv_{1,\phi}\cRic_\phi+\bigl(\cRic_\phi\bigr)^2$ and $\cE_{2,\phi}=\cT_{2,\phi}-\cv_{2,\phi}g$.
\end{lem}

\begin{proof}
 Suppose first that $t=0$.  By~\cite[Corollary~8.4]{Case2014sd}, it holds that
 \begin{equation}
  \label{eqn:v3dot1}
  \left(\tau^3\csigma_{3,\phi}\right)^\prime = \tau^3\delta_\phi\left(\cT_{2,\phi}(\nabla\psi_0)\right) + \frac{\tau^2}{2}\cv_{2,\phi}\psi_0 - \tau^3\left(d\cRic_\phi\cdot\cRic_\phi\right)(\nabla\psi_0) .
 \end{equation}
 Using the identity $\delta_\phi\Rm=d\cRic_\phi$, we directly compute that
 \[ \bigl(\cB_\phi\bigr)^\prime = 2d\cRic_\phi(\cdot,\nabla\psi_0,\cdot) . \]
 It follows that
 \begin{equation}
  \label{eqn:v3dot2}
  \left(\lp\cB_\phi,\cRic_\phi\rp\right)^\prime = 2\left(d\cRic_\phi\cdot\cRic_\phi\right)(\nabla\psi_0) + \lp\cB_\phi,\nabla^2\psi_0\rp .
 \end{equation}
 In particular, combining~\eqref{eqn:v3dot1} and~\eqref{eqn:v3dot2} with Lemma~\ref{lem:div_bach} yields
 \begin{equation}
  \label{eqn:v31}
  \left(\tau^3\cv_{3,\phi}\right)^\prime = \tau^3\delta_\phi\left(\left(\cT_{2,\phi}+\frac{1}{3}\cB_\phi\right)(\nabla\psi_0)\right) + \frac{\tau^2}{2}\cv_{2,\phi}\psi_0 .
 \end{equation}

 Suppose next that $\psi_0=0$.  Write
 \[ \tau^3\cv_{3,\phi} = \frac{1}{6}\left(\tau\cv_{1,\phi}\right)^3 - \frac{1}{2}\left(\tau\cv_{1,\phi}\right)\lv\tau\cRic_\phi\rv^2 + \frac{1}{3}\tr\left(\tau\cRic_\phi\right)^3 + \frac{\tau}{3}\lp\tau\cB_\phi,\tau\cRic_\phi\rp . \]
 One readily computes that $(\tau\cB_\phi)^\prime=tB_\phi$, from which it follows that
 \[ \left(\tau^3\cv_{3,\phi}\right)^\prime = t\tau^2\left(\cv_{2,\phi}\left(\sigma_{1,\phi}-\frac{n}{4\tau}\right) + \left\lp\cE_{2,\phi}+\frac{1}{3}\cB_\phi,\Ric_\phi\right\rp + \frac{1}{3}\lp\cB_\phi+B_\phi,\cRic_\phi\rp\right) . \]
 Combining this with~\eqref{eqn:v31} and the identity $\tau B_\phi=\tau\cB_\phi+\frac{1}{2}\Ric$ yields~\eqref{eqn:v3_local_var}.
\end{proof}

We now compute the first variation of the $\mV_3$-functional~\eqref{eqn:V3}.

\begin{thm}
 \label{thm:W3_first_variation}
 Let $(M^n,g,e^{-\phi}\dvol)$ be a manifold with density and let $\tau>0$ be such that $(\phi,\tau)\in\mC_1(g)$.  Then
 \begin{multline}
  \label{eqn:W3_first_variation}
   \mV_3^\prime[\psi,t] = -\int_M \bigg\{ \left(\tau^3\cv_{3,\phi} - \frac{\tau^2}{2}\cv_{2,\phi}\right)\psi_0 - t\tau^2\bigg[\cv_{2,\phi}\left(\sigma_{1,\phi}-\frac{n}{4\tau}\right) \\ + \left\lp\cE_{2,\phi}+\frac{1}{3}\cB_\phi,\Ric_\phi\right\rp + \frac{1}{6\tau}\left\lp 4\tau\cB_\phi + \Ric,\cRic_\phi\right\rp\bigg] \bigg\}\, d\nu ,
 \end{multline}
 for all $(\psi,t)\in T_{(\phi,\tau)}\mC_1(g)$, where $\psi_0=\psi+\frac{nt}{2\tau}$.
\end{thm}

\begin{proof}
 Let $\gamma\colon(-\eps,\eps)\to\mC_1(g)$ be a smooth path with $\gamma(0)=(\phi,\tau)$ and $\gamma^\prime(0)=(\psi,t)$.  Denote $\gamma(s)=\left(\phi(s),\tau(s)\right)$.  By definition,
 \[ \mV_3^\prime[\psi,t] = \left(\mV_3\left(g,\phi(s),\tau(s)\right)\right)^\prime(0) = \int_M \left(\tau^3\cv_{3,\phi}\right)^\prime d\nu + \int_M \tau^3\cv_{3,\phi}\,(d\nu)^\prime . \]
 The conclusion now follows from Lemma~\ref{lem:v3_local_all_var}.
\end{proof}

Theorem~\ref{thm:W3_first_variation} yields Theorem~\ref{thm:variational_theorem}.

\begin{proof}[Proof of Theorem~\ref{thm:variational_theorem}]
 Let $\gamma(s)=\left(\phi(s),\tau\right)$ be a smooth path in $\mC_1(g,\tau)$ with $\phi(0)=\phi$.  By Theorem~\ref{thm:W3_first_variation} and~\cite[Proposition~5.2]{Case2014sd},
 \[ \left.\frac{d}{ds}\right|_{s=0}\mW_3\left(g,\phi(s),\tau\right) = -\int_M \left[ \tau^3\cv_{3,\phi} - \frac{\tau^2}{2}\cv_{2,\phi} + \frac{\tau}{8}\cv_{1,\phi} - \frac{1}{16}\right]\psi\, d\nu , \]
 from which the conclusion readily follows.
\end{proof}

Theorem~\ref{thm:W3_first_variation} also yields the first assertion of Theorem~\ref{thm:stability_theorem}.

\begin{cor}
 \label{cor:grs_crit}
 Let $(M^n,g,e^{-\phi}\dvol)$ be a volume-normalized gradient Ricci soliton.  Then $(\phi,\tau)$ is a critical point of $\mW_3\colon\mC_1(g)\to\bR$.
\end{cor}

\begin{proof}
 Since the weighted Bach tensor vanishes for any gradient Ricci soliton, we have $\cv_{k,\phi}=\csigma_{k,\phi}$ for $k\in\{1,2,3\}$.  Thus, by Lemma~\ref{lem:grs_sigmak}, we see that $\cv_{k,\phi}$ is constant for $k\in\{1,2,3\}$.  The conclusion now follows from Theorem~\ref{thm:variational_theorem}.
\end{proof}
\section{The second variation of $\mW_3$}
\label{sec:stable}

In order to prove Theorem~\ref{thm:stability_theorem}, we first compute the second variation of the $\mV_3$-functional at a volume-normalized shrinking gradient Ricci soliton.

\begin{thm}
 \label{thm:V3_second_var}
 Let $(M^n,g,e^{-\phi}\dvol)$ be a volume-normalized shrinking gradient Ricci soliton.  Then
 \begin{equation}
  \label{eqn:V3_second_var} \mV_3^{\prime\prime}[\psi,t] \geq \left(\tau^2\cv_{2,\phi}-\frac{\tau}{2}\cv_{1,\phi}\right)\int_M \left[ \tau\left|\nabla\psi_1\right|^2 - \frac{1}{2}\psi_1^2 + \frac{1}{2}\left(c-\frac{t}{\tau}\right)^2\phi_0^2 \right] d\nu ,
 \end{equation}
 where $\psi_1=\psi_0+c\phi_0$ for $\psi_0=\psi+\frac{nt}{2\tau}$ and $c$ such that $\int\psi_1\phi_0\,d\nu=0$.  Moreover, equality holds in~\eqref{eqn:V3_second_var} if and only if $(M^n,g,e^{-\phi}\dvol)$ is isometric to a shrinking Gaussian or $t=0$.
\end{thm}

\begin{proof}
 Let $\gamma\colon(-\eps,\eps)\to C^\infty(M)\times(0,\infty)$ be a smooth path with $\gamma(0)=(\phi,\tau)$ and $\frac{\partial\gamma}{\partial s}=(\psi,t)$.  Define functions $A,B,C,D\colon (-\eps,\eps)\to\bR$ as follows:
 \begin{align*}
  A(s) & = -\int_M \left(\tau^3\cv_{3,\phi}-\frac{\tau}{2}\cv_{2,\phi}\right)\psi_0\,d\nu , \\
  B(s) & = \int_M \frac{t}{\tau}\left(\tau^2\cv_{2,\phi}\right)\left(\tau\sigma_{1,\phi}-\frac{n}{4}\right)d\nu , \\
  C(s) & = \int_M t\left\lp \tau^2\cE_{2,\phi} + \frac{\tau^2}{3}\cB_\phi, \Ric_\phi\right\rp d\nu , \\
  D(s) & = \int_M \frac{t}{6}\left\lp 4\tau\cB_\phi + \Ric, \tau\cRic_\phi\right\rp d\nu .
 \end{align*}
 Since $\left(\mV\left(g,\phi(s),\tau(s)\right)\right)^\prime=\left(A+B+C+D\right)(s)$, it suffices to compute the derivatives of $A,B,C$ and $D$ at $s=0$.  This task is made simpler by combining Lemma~\ref{lem:grs_sigmak}, Lemma~\ref{lem:v1_local_all_var}, Lemma~\ref{lem:v2_local_all_var} and Lemma~\ref{lem:v3_local_all_var} to conclude that
 \begin{equation}
  \label{eqn:grs_var_vk}
  \left(\tau^k\cv_{k,\phi}\right)^\prime(0) = \tau^{k-1}\cv_{k-1,\phi}\left(\tau\Delta_\phi + \frac{1}{2}\right)\psi_0 - \frac{t\tau^{k-2}}{2}\cv_{k-1,\phi}\phi_0
 \end{equation}
 for $k\in\{1,2,3\}$.  Furthermore, \cite[Corollary~8.4]{Case2014sd} implies that
 \begin{equation}
  \label{eqn:grs_var_s1}
  \left(\tau\sigma_{1,\phi}-\frac{n}{4}\right)^\prime(0) = \tau\Delta_\phi\psi_0 - \frac{t}{2\tau}\phi_0 + \frac{nt}{4\tau}
 \end{equation}
 and direct computation yields
 \begin{align}
  \label{eqn:grs_var_cRic} \left(\tau\cRic_\phi\right)^\prime(0) & = \tau\nabla^2\psi_0 + \frac{t}{2\tau}g , \\
  \label{eqn:grs_var_cE2} \left(\tau^2\cE_{2,\phi}\right)^\prime(0) & = -\tau^2\cv_{1,\phi}\nabla^2\psi_0 - \frac{t}{2}\cv_{1,\phi}g , \\
  \label{eqn:grs_var_cB} \left(\tau^2\cB_\phi\right)^\prime(0) & = -\frac{t}{2}\nabla^2\phi_0 + \frac{t}{4\tau}g .
 \end{align}

 By adapting the proof of Corollary~\ref{cor:grs_crit}, we see that $(\phi,\tau)$ is a critical point of $\mV_3$.  Hence
 \[ A^\prime(0) = -\int_M \left(\tau^3\cv_{3,\phi}-\frac{\tau}{2}\cv_{2,\phi}\right)^\prime(0)\,\psi_0\,d\nu . \]
 Therefore, by Lemma~\ref{lem:grs_sigmak} and~\eqref{eqn:grs_var_vk},
 \begin{equation}
  \label{eqn:Aprime}
  A^\prime(0) = \left(\tau^2\cv_{2,\phi}-\frac{\tau}{2}\cv_{1,\phi}\right)\int_M \left[ \tau\lv\nabla\psi_0\rv^2 - \frac{1}{2}\psi_0^2 + \frac{t}{2\tau}\phi_0\psi_0\right]d\nu .
 \end{equation}
 
 Using Lemma~\ref{lem:grs_sigmak} and Lemma~\ref{lem:estimates}, we observe that
 \begin{multline*}
  B^\prime(0) = -\frac{t}{2\tau}\int_M \left(\tau^2\cv_{2,\phi}\right)^\prime(0)\,\phi_0\,d\nu \\ + t\tau\cv_{2,\phi}\int_M \left(\tau\sigma_{1,\phi}-\frac{n}{4}\right)^\prime(0)\,d\nu - \frac{t\tau}{2}\cv_{2,\phi} \int_M \phi_0\,(d\nu)^\prime(0) .
 \end{multline*}
 From~\eqref{eqn:grs_var_vk} and \eqref{eqn:grs_var_s1} we compute that
 \[ B^\prime(0) = \int_M \left[ \frac{t}{2}\left(\tau\cv_{2,\phi}+\frac{1}{2}\cv_{1,\phi}\right)\psi_0\phi_0 + \frac{t^2}{4\tau}\cv_{1,\phi}\phi_0^2 + \frac{nt^2}{4}\cv_{2,\phi} \right]d\nu . \]
 Applying Lemma~\ref{lem:grs_sigmak} and Lemma~\ref{lem:estimates} yields
 \begin{equation}
  \label{eqn:Bprime}
  B^\prime(0) \geq \left(\frac{t\tau}{2}\cv_{2,\phi} + \frac{t}{4}\cv_{1,\phi}\right) \int_M \left[ \psi_0\phi_0 + \frac{t}{\tau}\phi_0^2 \right]d\nu
 \end{equation}
 with equality if and only if $(M^n,g,e^{-\phi}\dvol)$ is isometric to a shrinking Gaussian or $t=0$.

 Using first~\eqref{eqn:grs_var_cE2} and~\eqref{eqn:grs_var_cB} and second~\eqref{eqn:sigma1_estimate}, we compute that
 \begin{align*}
  C^\prime(0) & = -\frac{t}{2\tau}\int_M \tr\left( \tau^2\cv_{1,\phi}\nabla^2\psi_0 + \frac{t}{2}\cv_{1,\phi}g + \frac{t}{6}\nabla^2\phi_0 - \frac{t}{12\tau}g\right)d\nu \\
  & = -\frac{t}{2\tau}\int_M \left[ \tau\cv_{1,\phi}\psi_0\phi_0 + \frac{nt}{2}\cv_{1,\phi} + \frac{t}{6\tau}\phi_0^2 - \frac{nt}{12\tau} \right]d\nu .
 \end{align*}
 Lemma~\ref{lem:grs_sigmak} and Lemma~\ref{lem:estimates} then yield
 \begin{equation}
  \label{eqn:Cprime}
  C^\prime(0) \geq -\frac{t}{2}\cv_{1,\phi}\int_M \left[ \psi_0\phi_0 + \frac{t}{\tau}\phi_0^2 \right]d\nu
 \end{equation}
 with equality if and only if $(M^n,g,e^{-\phi}\dvol)$ is isometric to a shrinking Gaussian or $t=0$.

 From~\eqref{eqn:grs_var_cRic} we compute that
 \[ D^\prime(0) = \frac{t}{6}\int_M \left[ \frac{1}{2}\lp\nabla\psi_0,\nabla\phi_0\rp - \tau\lp\nabla^2\psi_0,\nabla^2\phi_0\rp - \frac{t}{2\tau^2}\phi_0^2 + \frac{nt}{4\tau^2} \right] d\nu . \]
 Recalling that $\delta_\phi\nabla^2u=d\Delta_\phi u + \Ric_\phi(\nabla u)$ for all $u\in C^\infty(M)$, we compute using Lemma~\ref{lem:estimates} that
 \[ \int_M \lp\nabla^2\psi_0,\nabla^2\phi_0\rp\,d\nu = \frac{1}{2\tau}\int_M \lp\nabla\psi_0,\nabla\phi_0\rp\,d\nu . \]
 Inserting this into the previous display and using Lemma~\ref{lem:estimates} again yields
 \begin{equation}
  \label{eqn:Dprime}
  D^\prime(0) \geq 0
 \end{equation}
 with equality if and only if $(M^n,g,e^{-\phi}\dvol)$ is isometric to a shrinking Gaussian or $t=0$.
 
 Now, combining~\eqref{eqn:Aprime}, \eqref{eqn:Bprime}, \eqref{eqn:Cprime} and~\eqref{eqn:Dprime} yields
 \[ \mV_3^{\prime\prime}[\psi,\tau] \geq \left(\tau^2\cv_{2,\phi} - \frac{\tau}{2}\cv_{1,\phi}\right)\int_M \left[ \tau\lv\nabla\psi_0\rv^2 - \frac{1}{2}\psi_0^2 + \frac{t}{\tau}\psi_0\phi_0 + \frac{t^2}{2\tau^2}\phi_0^2 \right] d\nu \]
 with equality if and only if $(M^n,g,e^{-\phi}\dvol)$ is isometric to a shrinking Gaussian or $t=0$.  The final conclusion follows from the observation that
 \begin{align*}
  \int_M \lv\nabla\psi_0\rv^2\,d\nu & = \int_M \left[ \lv\nabla\psi_1\rv^2 + \frac{c^2}{\tau}\phi_0^2\right] d\nu , \\
  \int_M \psi_0^2\,d\nu & = \int_M \left[ \psi_1^2 + c^2\phi_0^2 \right]d\nu . \qedhere
 \end{align*}
\end{proof}

\begin{proof}[Proof of Theorem~\ref{thm:stability_theorem}]
 It follows from Theorem~\ref{thm:V3_second_var} and the proof of~\cite[Theorem~9.2]{Case2014sd} that
 \[ \mW_3^{\prime\prime}[\psi,\tau] \geq \left(\tau^2\cv_{2,\phi} - \frac{\tau}{2}\cv_{1,\phi} + \frac{1}{8}\right)\int_M \left[ \tau\lv\nabla\psi_1\rv^2 - \frac{1}{2}\psi_1^2 + \frac{1}{2}\left(c-\frac{t}{\tau}\right)^2\phi_0^2 \right] d\nu \]
 with equality if and only if $(M^n,g,e^{-\phi}\dvol)$ is isometric to a shrinking Gaussian or $t=0$.  By Lemma~\ref{lem:grs_sigmak}, we see that $\tau^2\cv_{2,\phi}-\frac{\tau}{2}\cv_{1,\phi}+\frac{1}{8}>0$.  Hence $\mW_3^{\prime\prime}[\psi,\tau]\geq0$ with equality if and only if
 \begin{enumerate}
  \item $(M^n,g,e^{-\phi}\dvol)$ is isometric to a shrinking Gaussian, $c-\frac{t}{\tau}=0$ and $\int\lv\nabla\psi_1\rv^2\,d\nu=\frac{1}{2\tau}\int\psi_1^2\,d\nu$, or
  \item $t=0$, $\int\psi\phi_0\,d\nu=0$ and $\int\lv\nabla\psi\rv^2\,d\nu=\frac{1}{2\tau}\int\psi^2\,d\nu$.
 \end{enumerate}
 The conclusion follows from a straightforward application of Lemma~\ref{lem:grs_obata}.
\end{proof}


\bibliographystyle{abbrv}
\bibliography{../bib}
\end{document}